\documentclass[12pt]{article}
\usepackage{amsthm,amsmath}
\usepackage{amscd}
 \usepackage{amssymb}
 \usepackage{hyperref}
 \usepackage[all]{xypic}
 \usepackage{mathtools}
 \usepackage[utf8]{inputenc}
  \newtheorem{lemma}{Lemma}[section]
 \newtheorem{corollary}[lemma]{Corollary}
 
 \newtheorem{theorem}[lemma]{Theorem}
 \newtheorem{proposition}[lemma]{Proposition}
 
 \newtheorem{remark}[lemma]{Remark}

\usepackage[utf8]{inputenc} 

\def\includegraphics{}

\usepackage{tikz,xcolor,hyperref}

\definecolor{lime}{HTML}{A6CE39}
\DeclareRobustCommand{\orcidicon}{
	\begin{tikzpicture}
	\draw[lime, fill=lime] (0,0)
	circle[radius=0.16]
	node[white]{{\fontfamily{qag}\selectfont \tiny \.{I}D}};
	\end{tikzpicture}
	\hspace{-2mm}
}
\foreach \x in {A, ..., Z}{%
	\expandafter\xdef\csname orcid\x\endcsname{\noexpand\href{https://orcid.org/\csname orcidauthor\x\endcsname}{\noexpand\orcidicon}}
}

\usepackage{amscd}
\usepackage{amsmath,amssymb,mathrsfs}
\usepackage{hyperref}
\usepackage[all]{xypic}
\usepackage{times}
\usepackage{inputenc}
\usepackage{titlesec}

\allowdisplaybreaks[1]
\allowdisplaybreaks[4]

 \usepackage{amsthm}

\advance\headheight1.15pt

\newenvironment{proof of Theorem 3.1}{{\noindent\it {\bf Proof of Theorem 3.1}}\quad}{\hfill $\square$\par}
\newenvironment{proof of Theorem 2.13}{{\noindent\it {\bf Proof of Theorem 2.13}}\quad}{\hfill $\square$\par}
\newenvironment{proof of Theorem 4.1}{{\noindent\it {\bf Proof of Theorem 4.1}}\quad}{\hfill $\square$\par}
\newenvironment{proof of Theorem 1.4}{{\noindent\it {\bf Proof of Theorem 1.4}}\quad}{\hfill $\square$\par}
\newenvironment{proof of Theorem 1.5}{{\noindent\it {\bf Proof of Theorem 1.5}}\quad}{\hfill $\square$\par}

\newcommand{\dd}{\mathbb D}

\newcommand{\hd}{H(\dd)}
\newcommand{\hp}{H^p}         \newcommand{\hq}{H^{q}}
 \newcommand{\sn}{\sum_{n=0}^{\infty}}

\newcommand{\hg}{\mathcal H_g}
\newcommand{\h}{\mathcal H}

 \newcommand{\sk}{\sum_{k=0}^{\infty}}

\begin{document}

\baselineskip=8pt
\title{\fontsize{15}{0}\selectfont Generalized Hilbert operators on Hardy spaces}

\author{\fontsize{11}{0}\selectfont
	Yuting Guo$^{a}$  and Pengcheng Tang$^{*,b}$
	\\
	\fontsize{10}{0}\it{
		$^{a}$School of Mathematics and Statistics,
		Hunan First Normal University, Changsha, Hunan 410205, China} 
	\\ \fontsize{10}{0}\it{  $^{b}$School of Mathematics and Statistics, Hunan University of Science and Technology,}
	\\ \fontsize{10}{0}\it{ Xiangtan, Hunan 411201, China}
}

\date{}
\maketitle
\thispagestyle{empty}
\begin{center}

\textbf{\underline{ABSTRACT}}
\end{center}
\ \ \ \ \ Let $g\in H(\mathbb D)$,   the generalized Hilbert operator $\mathcal H_g$ is defined by
\[
\mathcal H_g(f)(z)=\int_0^1 f(t)g'(tz)dt,\ \ z\in \mathbb D\, \ \ f \in  H(\mathbb D).
\]
Let $\mathcal R_p=\mathcal H(H^p)$ be the range of the classical Hilbert operator on Hardy space, equipped with the pullback norm, and let $(\mathcal R_p,H^p)$ denote the  Hadamard multiplier space. For $1<p<\infty$, we prove the exact multiplier characterization
\[
\mathcal H_g:H^{p}\longrightarrow H^{p} \ \ \text{is bounded}
\quad\Longleftrightarrow\quad
g'\in(\mathcal R_p,H^p),
\]
and an equivalent Hilbert-matrix bilinear criterion $\mathfrak B_p(g)<\infty$. We identify the multiplier space completely when $1<p\le2$:
\[
(\mathcal R_p,H^p)=H\left(p,\infty,\frac1{p'}\right).
\]
For $p>2$, we prove that the   multiplier space $(\mathcal R_p,H^p)$  is  strictly contained in $H\left(p,\infty,\frac1{p'}\right)$.  This shows that  \(g\in \Lambda(p,1/p)\) does not imply that $\hg$ is bounded on \(H^p\),  giving a negative answer to the conjecture posed by Galanopoulos, Girela, Pel\'aez and Siskakis.

In addition, we  locate two previously known sufficient classes inside the multiplier space. This allow us obtain a complete coefficient characterization of $\hg$ on $H^{p}$ for  $g \in  H(\mathbb D) $ with nonnegative decreasing Taylor coefficients.
We then study the structure of $(\mathcal R_p,H^p)$. 
We show that  the multiplier spaces $(\mathcal{R}_p,H^p)$ form a strictly increasing family with respect to the exponent $p$.


\begin{flushleft}
{\bf{Keywords:}}   Generalized Hilbert operator. Hardy space.  Hadamard multiplier. Mixed norm space. 
\end{flushleft}
\begin{flushleft}
{\bf{MSC 2020:}}  47B35, 30H10, 30H20.
\end{flushleft}

\let\thefootnote\relax\footnote{$^*$Corresponding Author}
\let\thefootnote\relax\footnote{ Pengcheng Tang: www.tang-tpc.com@foxmail.com}
\let\thefootnote\relax\footnote{ Yuting Guo: rainting2770@163.com }
\vspace{1cm}
%

\section{Introduction} \label{Sec:Intro}

\ \ \  \ \  Let $\mathbb{D}=\{z\in \mathbb{C}:\vert z\vert <1\}$ denote the open unit disk of the complex plane $\mathbb{C}$ and let $H(\mathbb D)$ denote the space of analytic functions in the unit disc.

For $0<p<\infty$ and $f\in H(\mathbb D)$, set
\[
M_p(r,f)=\left(\frac1{2\pi}\int_0^{2\pi}|f(re^{i\theta})|^p\,d\theta\right)^{1/p}.
\]
The Hardy space $H^p$ consists of those $f$ for which
\[
\|f\|_{H^p}=\sup_{0<r<1}M_p(r,f)<\infty.
\]
Throughout this paper, $p'=p/(p-1)$ denotes the conjugate exponent.

The mixed norm space $H(p,q,\alpha)$, $0<p,q\leq \infty$, $0<\alpha<\infty$,  is the space of all
functions $f\in \hd$, for which
$$
||f||_{p,q,\alpha}=\left(\int_{0}^{1}M^{q}_{p}(r,f)(1-r)^{q\alpha-1}dr\right)^{\frac{1}{q}}<\infty, \ \mbox{for}\ 0<q<\infty,
$$
and
$$||f||_{p,\infty,\alpha}=\sup_{0\leq r<1}(1-r)^{\alpha}M_{p}(r,f)<\infty.$$

For $f(z)=\sn a_{n}z^{n}\in H(\dd)$,
the fractional derivative  is defined by
 $$Df(z)=\sum_{n=0}^{\infty}(n+1)a_{n}z^{n}.$$

For $0<p,q\leq \infty$ and $0<\alpha<\infty$, we use  $\mathcal{D}(p,q,\alpha)$ to denote   the space of all analytic
functions $f\in \hd$ such that
$$||Df||_{p,q,\alpha}<\infty.$$

The  folliwing  embeddings can be found, for example,\cite{b1,f2}.
\begin{equation}\label{eq1}
	\mathcal D(p,p,1)\subset H^p\subset\mathcal D(p,2,1),\qquad 1<p\le2,
\end{equation}
and
\begin{equation}\label{eq2}
	\mathcal D(p,2,1)\subset H^p\subset\mathcal D(p,p,1),\qquad 2\le p<\infty.
\end{equation}

The theory of the spaces  $H(p,q,\alpha)$ was originated due to the work of Hardy
and Littlewood \cite{hl} and continued by Flett \cite{f2}. Such spaces arise naturally in the study of coefficient  multipliers on Hardy and weighted Bergman spaces.
Let $X $ and $Y $ be two spaces of analytic functions on the unit disc $\dd$. Let  $f(z)=\sn a_{n}z^{n}\in X$  and $\lambda:=\{\lambda_{n}\}_{n=0}^{\infty}$ be a sequence.
We can define the  multiplier operator  $T_{\lambda}$ as follows,
\begin{equation*}(T_{\lambda}f)(z)=\sn \lambda_{n}a_{n}z^{n}.\end{equation*}
If $T_{\lambda}:X\rightarrow Y$ , then $\lambda$
is said to be a coefficient multiplier or simply multiplier
from $X $ into $Y$. The set of all multipliers from $X$ to $Y$ is denoted by $(X, Y)$.  The   multipliers are closely related to the Hadamard product.

Recall that for $f(z)=\sum_{n=0}^{\infty}a_{n}z^{n}\in \hd$ and $g(z)=\sn b_{n}z^{n}\in \hd$, the Hadamard product of  $f,g\in \hd$ is defined by
$$(f\ast g)(z)=\sum_{n=0}^{\infty}a_{n}b_{n}z^{n}.$$

Multipliers on Hardy spaces were in fashion for a long time and much work was
done on them and related spaces. However,  complete descriptions of
multipliers between Hardy spaces $(\hp,\hq)$ for certain values of $p$  and  $q$ remain still open.  Many results on multipliers of Hardy space $\hp$,  mixed norm space $H^{p,q,\alpha}$ have been established in the last decades. We  refer the readers to the monograph \cite{jev} for more on multipliers and the theory of mixed norm spaces.

 The  Hilbert operator $\mathcal {H}$ defined on $\hd$ as follows:
 If $f(z)=\displaystyle{\sum_{n=0}^{\infty}}a_{n}z^{n}\in \hd$,  then
 \begin{equation*}\label{eq1}
 	\mathcal {H}(f)(z)=\sum_{n=0}^{\infty}\left(\sum_{k=0}^{\infty}\frac{a_{k}}{n+k+1}\right)z^{n}, \ \ z\in \dd,
 \end{equation*}
 whenever the right hand side makes sense and defines an analytic function in $\dd$.
 
 The study of the Hilbert operator $\mathcal{H}$ on  analytic function spaces was initiated by Diamantopoulos and Siskakis in \cite{dia1}. They proved that  $\mathcal {H}:H^{p}\rightarrow H^{p}$ is bounded for $1<p<\infty$ and  $\h$ is not bounded on $H^{1}$.  Subsequently, Diamantopoulus \cite{dia2} also considered the boundedness of $\mathcal {H}$  on the Bergman spaces $A^{p}$. He proved that $\h:A^{p}\rightarrow A^{p}$ is bounded for $2<p<\infty$ and $\h$ is not bounded on $A^{2}$.
 Jevti\'{c} and Karapretovi\'{c} \cite{jk}  investigated  the boundedness of $\h$ on mixed-norm spaces.  The reader is referred to \cite{bh,bell, bok,dos,lnp} for more about  Hilbert operator $\h$ on spaces of analytic functions.
 
 Let $\mathcal R_p:=\mathcal H(H^p)$
 the range of $\h$ on Hardy space $\hp$, equipped with the pullback norm
 \[
 \|\mathcal H(f)\|_{\mathcal R_p}:=\|f\|_{H^p}.
 \]

 For $g\in H(\mathbb D)$, the generalized Hilbert operator introduced in \cite{ann} is
 \begin{equation}\label{eq3}
 	\mathcal H_g(f)(z)=\int_0^1 f(t)g'(tz)\,dt,
 	\qquad z\in\mathbb D.
 \end{equation}
If $g(z)=\log\frac1{1-z}$, then $\hg$ is the classical Hilbert matrix operator.

For any $g\in \hd$, if $f\in \hd $ and $\int_{0}^{1}|f(t)|dt<\infty$, then the integral  in (\ref{eq3}) converges absolutely, and hence $\hg(f)$ is a well defined analytic function in $\dd$. Moreover, $\hg(f)$ has the following expression in terms of Hadamard product:
\begin{align}\label{eq4}
	\hg(f)(z)&=\sn \left((n+1)b_{n+1}\int_{0}^{1}t^{n}f(t)dt\right)z^{n} \nonumber\\
	& = \sn \left((n+1)b_{n+1}\sk \frac{a_{k}}{n+k+1}\right)z^{n}\nonumber\\
	& = g' \ast \h (f).
\end{align}

In  \cite{ann}, the authors  studied the boundedness of $\hg$ on the Hardy spaces $H^{p}$, the Bergman spaces  $A^{p}_{\alpha}$ and on the Dirichlet type spaces $\mathcal{D}_{\alpha}^{p}$.
Pel\'{a}ez and  R\"{a}tty\"{a} \cite{pj} also investigated the generalized Hilbert operator  $\hg$
acting on the weighted Bergman space  $A^{p}_{\omega}$, where $\omega$  belongs to a specific class of regular radial weight functions.
 Galanopoulos and Girela  used the Suchur test in  \cite{rra} to establish the boundedness and compactness of $\hg$ on  the Dirichlet space. The second author of this paper  also investigated the operator  $\hg$ on some spaces of analytic functions such as Dirichlet type space, space of bounded analytic functions, logarithmic Bloch spaces, see \cite{tt}.  The mean Lipschitz space plays a foundational role in these works.

Recall that for $1<p<\infty$ and $0<\alpha<1$, the mean Lipschitz space $\Lambda(p,\alpha)$ admits the following characterization
\begin{equation*}\label{eq:mean-lipschitz-derivative}
	g\in\Lambda(p,\alpha)
	\quad\Longleftrightarrow\quad
	g\in H^p\ \text{and}\ M_p(r,g')=O\bigl((1-r)^{\alpha-1}\bigr).
\end{equation*}
In particular,
\begin{equation*}\label{eq:Lambda-growth}
	g\in\Lambda\left(p,\frac1p\right)
	\quad\Longleftrightarrow\quad
	g'\in H\left(p,\infty,\frac1{p'}\right).
\end{equation*}

 In  \cite{ann}, Galanopoulos, Girela, Pel\'aez and Siskakis proved that, for $1<p\le2$,
 \[
 \mathcal H_g:H^p\to H^p \text{ is bounded}
 \quad\Longleftrightarrow\quad
 g\in\Lambda\left(p,\frac1p\right).
 \]
 For $p>2$, they proved the necessity of $g\in\Lambda(p,1/p)$ and the sufficient condition $g\in\Lambda(q,1/q)$ for some $1<q<p$, and conjectured that the endpoint condition $g\in\Lambda(p,1/p)$ should also be sufficient; see \cite[Theorems 1 and 2]{ann}.

In this paper, we shall prove that the elementary identity $\mathcal H_g(f)=g'*\mathcal H(f)$
yields the exact equivalence
\[
\mathcal H_g:H^p\to H^p \ \text{is bounded}
\quad\Longleftrightarrow\quad
g'\in(\mathcal R_p,H^p).
\]
This reformulation separates the operator from the symbol and turns the problem into the study of a concrete coefficient multiplier space.

Our  main result identifies the multiplier space when $1<p\le2$ and shows that the identification fails sharply for $p>2$:
\[
(\mathcal R_p,H^p)=H\left(p,\infty,\frac1{p'}\right),\qquad 1<p\le2,
\]
whereas
\[
(\mathcal R_p,H^p)\subsetneq H\left(p,\infty,\frac1{p'}\right),\qquad p>2.
\]
The strict inclusion for $p>2$ is verified by the following Hadamard gap  lacunary series
\[
\Phi_p(z)=\sum_{j=2}^\infty 2^{j/p'}z^{2^j}.
\]
We construct a function $f_p\in H^p$ satisfying $\Phi_p*\mathcal H(f_p)\notin H^p$. By integrating $\Phi_p$ , we obtain a symbol $g_p\in\Lambda(p,1/p)$ such that the  operator $\mathcal H_{g_p}$ is unbounded on $H^{p}$. This yields a negative  answer to the conjecture proposed in \cite{ann}.

It turns out that the multiplier space $(\mathcal R_p,H^p)$ play a  crucial role when investigating the boundedness of the  operator  $\hg$ on Hardy spaces. Accordingly, exploring the fundamental properties of multiplier space  $(\mathcal R_p,H^p)$ constitutes the core content of the subsequent sections of this paper. We shall prove the multiplier spaces $(\mathcal{R}_p,H^p)$ form a strictly increasing family with respect to the exponent $p$. We also locate two previously known sufficient classes inside the multiplier space. 
More precisely, for $p>2$, set
\[
\mathcal A_p=\bigcup_{1<q<p}H\left(q,\infty,\frac1{q'}\right),
\qquad
\frac1{t_p}=\frac12-\frac1p, \qquad	\mathcal E_p=H\left(p,t_p,\frac1{p'}\right). 
\]
We  will prove  that $\mathcal A_p\subsetneq(\mathcal R_p,H^p)$. This  provides an alternative proof of Theorem 2 in \cite{ann},  and  also demonstrates that the condition  $g\in \bigcup_{1<q<p}\Lambda(q,1/q)$ is not necessary for the boundedness of  $\hg:H^{p}\longrightarrow \hp$. We also show that $\mathcal E_p\subsetneq(\mathcal R_p,H^p)$. This  not only complements the sufficient condition obtained by Blasco in \cite{Blas2022}, but also provides an explicit proper subspace of $(\mathcal R_p,H^p)$ in a sharp sense.

This paper is organized as follows. In Section 2,  we collect some basic lemmas that will be used in our proof. In  Section 3, 
we  prove the multiplier and bilinear characterizations. The exact identification for $1<p\le2$ and the strict inclusion for $p>2$ are proved in this section. In  Section 4, we locate two previously known sufficient classes inside the multiplier space. We also treats symbols with non-negative decreasing coefficients and gives coefficient consequences.
Section 5 mainly  study the   basic  properties of the  multiplier space $(\mathcal R_p,H^p)$.

Throughout the paper, the letter $C$ will denote an absolute constant whose value depends on the parameters
indicated in the parenthesis, and may change from one occurrence to another. We will use
the notation $``P\lesssim Q"$ if there exists a constant $C=C(\cdot) $ such that $`` P \leq CQ"$, and $`` P \gtrsim Q"$ is
understood in an analogous manner. In particular, if  $``P\lesssim Q"$  and $ ``P \gtrsim Q"$ , then we will write $``P\asymp Q"$.

\section{Preliminaries}\label{sec2}

For $f(z)=\sum_{k=0}^\infty a_kz^k$, define
\[\Delta_nf(z)=\sum_{k=2^n}^{2^{n+1}-1}a_kz^k,
\qquad n\ge0.\]

For $p>1$ and $\alpha>0$, we  will use the  following  dyadic characterization
\begin{equation}\label{eq7}
	f\in H(p,\infty,\alpha)
	\quad\Longleftrightarrow\quad
	\sup_{n\ge0}2^{-n\alpha}\|\Delta_nf\|_{H^p}<\infty.
\end{equation}
More generally, for $1<p<\infty$, $0<t\leq\infty$, $\alpha>0$,
\begin{equation}\label{eq:dyadic-mixed}
	\|f\|_{H(p,q,\alpha)}^q
	\asymp |f(0)|^q+\sum_{n=0}^\infty
	\left(2^{-n\alpha}\|\Delta_nf\|_{H^p}\right)^q.
\end{equation}
These facts follow from the smooth dyadic decomposition of analytic mixed norm spaces; see \cite{mat,pm}.

We shall  use the following  Stieltjes--Mellin inequality.

\begin{lemma}\label{lem2.1}
	Let $1<p<\infty$ and $a\geq0$. If $u\in L^{p}(0,1)$, then
	\begin{equation*}\label{eq:stieltjes}
		\int_0^1x^{ap}\left|\int_0^1\frac{u(y)}{(x+y)^{1+a}}\,dy\right|^pdx
		\lesssim \int_0^1|u(y)|^p\,dy.
	\end{equation*}
\end{lemma}

\begin{proof}
	Extend $u$ by zero to $(1,\infty)$, so that
	$u\in L^p(0,\infty)$, and define
	\[
	S_a u(x)
	=
	x^a\int_0^\infty
	\frac{u(y)}{(x+y)^{1+a}}\,dy ,
	\qquad x>0 .
	\]
	Making the change of variables $y=xs$, we obtain
	\[
	S_a u(x)
	=
	x^a\int_0^\infty
	\frac{u(xs)x\,ds}{(x+xs)^{1+a}}
	=
	\int_0^\infty
	\frac{u(xs)}{(1+s)^{1+a}}\,ds .
	\]
	Hence, by Minkowski's integral inequality,
	\begin{align*}
		\|S_a u\|_{L^p(0,\infty)}
		&\leq
		\int_0^\infty
		\frac{\|u(xs)\|_{L_x^p(0,\infty)}}{(1+s)^{1+a}}\,ds .
	\end{align*}
	For $s>0$, a change of variables $t=xs$ gives
	\[
	\|u(xs)\|_{L_x^p(0,\infty)}
	=
	s^{-1/p}\|u\|_{L^p(0,\infty)}.
	\]
	Therefore,
	\begin{align*}
		\|S_a u\|_{L^p(0,\infty)}
		&\leq
		\|u\|_{L^p(0,\infty)}
		\int_0^\infty
		\frac{s^{-1/p}}{(1+s)^{1+a}}\,ds .
	\end{align*}
	The last integral is finite. Indeed, near the origin the integrand
	behaves like $s^{-1/p}$, which is integrable since $p>1$, while at
	infinity it behaves like $s^{-1/p-1-a}$,
	which is integrable for every $a\geq0$.
	
	Thus,
	\[	\|S_a u\|_{L^p(0,\infty)}	\lesssim	\|u\|_{L^p(0,\infty)}.	\]
	Restricting the $x$-integration to $(0,1)$ and using the fact that
	$u$ vanishes on $(1,\infty)$, we obtain, for $a\geq0$,
	\[\int_0^1	\left|	x^a	\int_0^1	\frac{u(y)}{(x+y)^{1+a}}\,dy	\right|^p dx
	\lesssim	\int_0^1 |u(y)|^p\,dy .	\]
	Equivalently,
	\[	\int_0^1 x^{ap}
	\left|	\int_0^1	\frac{u(y)}{(x+y)^{1+a}}\,dy	\right|^p dx
	\lesssim	\int_0^1 |u(y)|^p\,dy .	\]
	This proves the lemma.
\end{proof}

We need  the following lemma, which can be found in \cite[Page 232, Theorem 4.14]{zg}.

\begin{lemma}\label{lem2.2}
	Let  $1<p<\infty$ and let  $f(z)=\sn a_{n}z^{n}\in H^{p}$. 	Suppose that $\{\lambda_{n}\}_{n=0}^{\infty}$ is a sequence of complex numbers such that 
	$$|\lambda_{n}| \lesssim 1, \ \  \sum_{k=2^{n}}^{2^{n+1}-1}|\lambda_{k+1}-\lambda_{k}| \lesssim 1, \ \ n=0,1,2,\cdots.$$
	Then there exists a positive constant  $C_{p}$ such that 
	$$h(z)=\sn \lambda_{n}a_{n}z^{n}\in H^{p} \ \ \ \mbox{and}\ \ ||h||_{H^{p}}\leq C_{p}||f||_{H^{p}}.$$
\end{lemma}

	 The following result is essential and can be
	 proved by modifying the proof of Theorem 5.5 in  \cite{b1}.
	 
	\begin{lemma}\label{lem2.3}
		Let $1<p<\infty$, $h\in H(\mathbb D)$, and $0<r<R<1$. Then
		\[
		M_p(r,h')\le \frac{C_p}{R-r}M_p(R,h).
		\]
		In particular, if $h\in H(p,\infty,\alpha)$, then
		\[
		M_p(r,h')\lesssim 
		\|h\|_{H(p,\infty,\alpha)}(1-r)^{-1-\alpha}.
		\]
	\end{lemma}

	 We also need  the following embedding, see \cite{b1,mat}.
	 
	 \begin{lemma}\label{lem2.4}
	 	Let $1<q<p<\infty$. Then 
	 	\begin{equation*}\label{eq:sobolev-q-p}
	 		\|f\|_{H^p}^p
	 		\lesssim |f(0)|^p+
	 		\int_0^1(1-r)^{p/q'}M_q(r,f')^p\,dr.
	 	\end{equation*}
	 \end{lemma}
	 
	\section{The Hilbert-range multiplier space}\label{sec3}

	 We begin with the range of Hilbert operator on Hardy space.
	 \begin{proposition}\label{pro3.1}
	 Let $1<p<\infty$. 	Then the Hilbert operator $\mathcal H:H^p\to H^p$ is bounded and injective. Consequently,
	 	\[
	 	\|\mathcal H(f)\|_{\mathcal R_p}:=\|f\|_{H^p}
	 	\]
	 	defines a Banach norm on $\mathcal R_p$, and the inclusion $\mathcal R_p\hookrightarrow H^p$ is continuous.
	 \end{proposition}
	 
	 \begin{proof}
	 
	 	The boundedness of  $\mathcal H:H^p\to H^p$ follows from Theorem 1.1 in \cite{dia1}.
	 
	 	Let $f\in H^{p}$. Assume that $\mathcal H(f)=0$.  Since the Taylor coefficients of  $\mathcal H(f)=0$ are given by
	 	$$\int_{0}^1 t^{n}f(t)dt,$$
	 we have 
	 	\[
	 	\int_0^1t^nf(t)\,dt=0,
	 	\qquad n=0,1,2,\ldots.
	 	\]
	 This imples that
	 	\[
	 	\int_0^1P(t)f(t)\,dt=0
	 	\]
	 	for every polynomial $P$. By the Fej\'er--Riesz inequality,   $f|_{(0,1)}\in L^p(0,1)$.  Since polynomials are dense in  $L^{p'}(0,1)$, the continuous functional
	 	\[
	 	u\longmapsto\int_0^1u(t)f(t)\,dt
	 	\]
	 	vanishes on all  $L^{p'}(0,1)$. Therefore, $f=0$ almost everywhere on $(0,1)$. Since $f$ is continuous, it follows that $f\equiv 0$ on $(0,1)$. The identity theorem for analytic functions  gives $f\equiv 0$ in  $\dd$. This means that $\h$ is injective on $H^{p}$.
	 	
	 	Thus, the pullback norm is well defined. The definition of $\mathcal R_p$ shows that  $\mathcal H:H^p\to\mathcal R_p$ is an isometric isomorphism, so the  completeness follows. The continuous inclusion $\mathcal R_p\hookrightarrow H^p$ follows from the boundedness of $\mathcal H:H^p\to H^p$.
	 \end{proof}

	To simplify notations, we write
	 \[
	 \mathfrak M_p=(\mathcal R_p,H^p),
	 \qquad
	 \|\varphi\|_{\mathfrak M_p}
	 =\sup_{0\ne F\in\mathcal R_p}
	 \frac{\|\varphi*F\|_{H^p}}{\|F\|_{\mathcal R_p}}.
	 \]
	 
	Let
	\[
	g(z)=\sum_{n=0}^\infty b_nz^n,
	\qquad
	g'(z)=\sum_{k=0}^\infty\gamma_kz^k,
	\qquad
	\gamma_k=(k+1)b_{k+1}.
	\]
	For analytic polynomials
	\[
	f(z)=\sum_{n=0}^Na_nz^n,
	\qquad
	h(z)=\sum_{k=0}^Mc_kz^k,
	\]
set 
	\begin{equation*}\label{eq:B-form}
		\mathcal B_g(f,h)
		=\sum_{n=0}^N\sum_{k=0}^M
		\frac{\gamma_ka_n\overline{c_k}}{n+k+1},
	\end{equation*}
	and define
	\begin{equation*}\label{eq:Bnorm}
		\mathfrak B_p(g)
		=\sup_{f,h\ne0}
		\frac{|\mathcal B_g(f,h)|}
		{\|f\|_{H^p}\|h\|_{H^{p'}}},
	\end{equation*}
	where the supremum is taken over nonzero analytic polynomials.

	 \begin{theorem}\label{th3.2}
	 	Let $1<p<\infty$ and $g\in H(\mathbb D)$. The following  statements are equivalent:
	 
	 	 (i) The operator  $\mathcal H_g:H^p\to H^p$ is bounded;
	 	 
	 	(ii) $g'\in\mathfrak M_p=(\mathcal R_p,H^p)$;
	 	
	 	 (iii) $\mathfrak B_p(g)<\infty$.
	 
	 	Moreover,
	 	\begin{equation*}\label{eq:norm-main}
	 		\|\mathcal H_g\|_{H^p\to H^p}
	 		=\|g'\|_{\mathfrak M_p}
	 		\asymp \mathfrak B_p(g).
	 	\end{equation*}
	 \end{theorem}
	 
	 	\begin{proof}
	 		
	 		We first prove the equivalence between (i) and (ii).
	 		
	 		\medskip
	 		
	 		$(ii)\Longrightarrow(i)$.
	 		Assume that
	 	$	g'\in\mathfrak M_p=(\mathcal R_p,H^p).$
	 		For every $f\in H^p$, using the identity
	 		\[	\mathcal H_g(f)=g'*\mathcal H(f),	\]
	 		we obtain
	 		\begin{align*}
	 			\|\mathcal H_g(f)\|_{H^p}
	 			&=\|g'*\mathcal H(f)\|_{H^p}\\
	 			&\leq
	 			\|g'\|_{\mathfrak M_p}
	 			\|\mathcal H(f)\|_{\mathcal R_p}\\
	 			&=
	 			\|g'\|_{\mathfrak M_p}\|f\|_{H^p}.
	 		\end{align*}
	 		Hence $\mathcal H_g$ is bounded and
	 		\[
	 		\|\mathcal H_g\|_{H^p\to H^p}
	 		\leq
	 		\|g'\|_{\mathfrak M_p}.
	 		\]
	 		
	 		\medskip
	 		
	 		$(i)\Longrightarrow(ii)$.
	 		Assume that $\mathcal H_g:H^p\to H^p$ is bounded. Let
	 		$F\in\mathcal R_p$. By the definition of $\mathcal R_p$,
	 		there exists $f\in H^p$ such that
	 		\[
	 		F=\mathcal H(f).
	 		\]
	 		The injectivity of $\mathcal H$ implies that such an $f$ is unique.
	 		Therefore,
	 		\[
	 		g'*F
	 		=
	 		g'*\mathcal H(f)
	 		=
	 		\mathcal H_g(f).
	 		\]
	 		Consequently,
	 		\begin{align*}
	 			\|g'*F\|_{H^p}
	 			&=
	 			\|\mathcal H_g(f)\|_{H^p}\\
	 			&\leq
	 			\|\mathcal H_g\|_{H^p\to H^p}
	 			\|f\|_{H^p}\\
	 			&=
	 			\|\mathcal H_g\|_{H^p\to H^p}
	 			\|F\|_{\mathcal R_p}.
	 		\end{align*}
	 		Hence
	 		\[	g'\in\mathfrak M_p	\]
	 		and
	 		\[	\|g'\|_{\mathfrak M_p}	\leq	\|\mathcal H_g\|_{H^p\to H^p}.	\]
	 		Thus, we obtain
	 		\[	\|\mathcal H_g\|_{H^p\to H^p}	=	\|g'\|_{\mathfrak M_p}.	\]
	 		
	 		It remains to prove the equivalence with (iii).

	 		$(i)\Longrightarrow(iii)$.
	 		Assume that $\mathcal H_g$ is bounded. Let
	 		\[
	 		f(z)=\sum_{n=0}^{N}a_nz^n,
	 		\qquad
	 		h(z)=\sum_{k=0}^{M}c_kz^k
	 		\]
	 		be analytic polynomials. Since
	 		\[
	 		\mathcal H(f)(z)
	 		=
	 		\sum_{k=0}^{\infty}
	 		\left(
	 		\sum_{n=0}^{N}
	 		\frac{a_n}{n+k+1}
	 		\right)z^k,
	 		\]
	 		we have
	 		\[
	 		\mathcal H_g(f)(z)
	 		=
	 		\sum_{k=0}^{\infty}
	 		\gamma_k
	 		\left(
	 		\sum_{n=0}^{N}
	 		\frac{a_n}{n+k+1}
	 		\right)z^k .
	 		\]
	 		Hence, using the Hardy space dual pairing between $H^p$ and
	 		$H^{p'}$ we have
	 		\begin{align*}
	 			\langle\mathcal H_g(f),h\rangle
	 			&=
	 			\sum_{k=0}^{M}
	 			\gamma_k
	 			\left(
	 			\sum_{n=0}^{N}
	 			\frac{a_n}{n+k+1}
	 			\right)
	 			\overline{c_k}\\
	 			&=
	 			\sum_{n=0}^{N}\sum_{k=0}^{M}
	 			\frac{\gamma_k a_n\overline{c_k}}
	 			{n+k+1}\\
	 			&=
	 			\mathcal B_g(f,h).
	 		\end{align*}
	 		By the duality of Hardy spaces, we have
	 		\begin{align*}
	 			|\mathcal B_g(f,h)|
	 			&=
	 			|\langle\mathcal H_g(f),h\rangle|\\
	 			&\leq
	 			C_p
	 			\|\mathcal H_g(f)\|_{H^p}
	 			\|h\|_{H^{p'}}\\
	 			&\leq
	 			C_p
	 			\|\mathcal H_g\|
	 			\|f\|_{H^p}
	 			\|h\|_{H^{p'}} .
	 		\end{align*}
	 		Taking the supremum over all nonzero analytic polynomials $f,h$, we obtain 
	 		\[	\mathfrak B_p(g)	\lesssim	\|\mathcal H_g\|_{H^p\to H^p}.\]

	 		$(iii)\Longrightarrow(i)$.
	 		Assume that	$	\mathfrak B_p(g)<\infty.$
	 		Let
	 		\[
	 		f(z)=\sum_{n=0}^{N}a_nz^n
	 		\]
	 		be an analytic polynomial. For $M\ge0$, define
	 		\[
	 		T_Mf(z)
	 		=
	 		\sum_{k=0}^{M}
	 		\gamma_k
	 		\left(
	 		\sum_{n=0}^{N}
	 		\frac{a_n}{n+k+1}
	 		\right)z^k .
	 		\]
	 		Then $T_Mf$ is the $M$-th Taylor partial sum of $\mathcal H_g(f)$.
	 		
	 		For every analytic polynomial
	 		\[
	 		h(z)=\sum_{k=0}^{M}c_kz^k ,
	 		\]
	 		we have
	 		\[
	 		\langle T_Mf,h\rangle
	 		=	\mathcal B_g(f,h).
	 		\]
	 		Hence,
	 		\[
	 		|\langle T_Mf,h\rangle|
	 		\leq
	 		\mathfrak B_p(g)
	 		\|f\|_{H^p}
	 		\|h\|_{H^{p'}}.
	 		\]
	 		By the duality of Hardy space,
	 		\[
	 		\|T_Mf\|_{H^p}
	 		\lesssim
	 		\mathfrak B_p(g)
	 		\|f\|_{H^p},
	 		\]
	 		where the constant is independent of $M$.
	 		
	 		Now fix $0<r<1$. Since $T_Mf$ are the Taylor partial sums of
	 		$\mathcal H_g(f)$, we have
	 		\[
	 		T_Mf(r\cdot)\longrightarrow
	 		\mathcal H_g(f)(r\cdot)
	 		\]
	 		uniformly on $\overline{\mathbb D}$.
	 		Therefore, by Fatou's lemma,
	 		\begin{align*}
	 			M_p(r,\mathcal H_g(f))
	 			&\leq
	 			\liminf_{M\to\infty}
	 			\|T_Mf(r\cdot)\|_{H^p}\\
	 			&\leq
	 			\liminf_{M\to\infty}
	 			\|T_Mf\|_{H^p}\\
	 			&\lesssim
	 			\mathfrak B_p(g)
	 			\|f\|_{H^p}.
	 		\end{align*}
	 		The constant is independent of $r$. Taking the supremum over
	 		$0<r<1$ gives
	 		\[\|\mathcal H_g(f)\|_{H^p}	\lesssim	\mathfrak B_p(g)	\|f\|_{H^p}.	\]
	 		Thus $\mathcal H_g$ is bounded on analytic polynomials. Since analytic polynomials are dense in $H^p$, $\mathcal H_g$ extends uniquely to a	bounded operator on $H^p$. Hence
	 		\[	\|\mathcal H_g\|_{H^p\to H^p}	\lesssim	\mathfrak B_p(g).\]
	 		
	 		Combining this estimate with the previous inequality yields
	 		$$
	 		\|\mathcal H_g\|_{H^p\to H^p}	\asymp	\mathfrak B_p(g).	$$
	 		This completes the proof.
	 	\end{proof}

	\begin{theorem}\label{th3.3}
		For $1<p<\infty$,
		\begin{equation*}\label{eq:universal-inclusion}
			\mathfrak M_p\subset H\left(p,\infty,\frac1{p'}\right).
		\end{equation*}
		More precisely,
		\[
		\|\varphi\|_{H(p,\infty,1/p')}
		\lesssim \|\varphi\|_{\mathfrak M_p}.
		\]
	\end{theorem}
	
	\begin{proof}
		Write
		\[	\varphi(z)=\sum_{k=0}^{\infty}c_kz^k.	\]
		
		For $N\ge2$, let
		\[
		K_N(z)=\sum_{j=0}^{N-1}z^j,
		\qquad F_N=\mathcal H(K_N).
		\]
		The analytic Dirichlet-kernel estimate gives
		\[	\|K_N\|_{H^p}\asymp N^{1/p'}.	\]
	 By the definition of the pullback norm,
		\[
		\|F_N\|_{\mathcal R_p}=\|K_N\|_{H^p}\asymp N^{1/p'}.
		\]
		The $k$-th coefficient of $F_N$ is
		\[	\lambda_{N,k}	=\int_0^1t^kK_N(t)\,dt	=\sum_{j=0}^{N-1}\frac1{k+j+1}.	\]
		
		If $N\le k<2N$, then every denominator $k+j+1$ lies between $N+1$ and $3N$, hence
		\[
		\frac13\le\lambda_{N,k}\le1.
		\]
		Moreover, $k\mapsto\lambda_{N,k}$ is decreasing. Thus
		\[
		\mu_{N,k}=\lambda_{N,k}^{-1},
		\qquad N\le k<2N,
		\]
		is increasing,  satisfies $1\le\mu_{N,k}\le3$, and has total variation at most $2$. Applying Lemma \ref{lem2.2}  to the polynomial
		\[
		P_{2^{n}}(z)=\sum_{k=2^{n}}^{2^{n+1}-1}\lambda_{2^{n},k}c_kz^k
		\]
	we obtain
		\[	\left\|\sum_{k=2^{n}}^{2^{n+1}-1}c_kz^k\right\|_{H^p} \lesssim\|P_{2^{n}}\|_{H^p}.	\]
		The polynomial $P_{2^{n}}$ is the block projection of $\varphi*F_{2^{n}}$. Since block projections are uniformly bounded on $H^p$,
		\begin{align*}
			\left\|\sum_{k=2^{n}}^{2^{n+1}-1}c_kz^k\right\|_{H^p}
			&\lesssim \|\varphi*F_{2^{n}}\|_{H^p}\\
			&\le \|\varphi\|_{\mathfrak M_p}\|F_{2^{n}}\|_{\mathcal R_p}
			\lesssim 2^{n/p'} \|\varphi\|_{\mathfrak M_p}.
		\end{align*}
		
	Thus,
	\[	\|\Delta_n\varphi\|_{H^p}\lesssim 2^{n/p'}\|\varphi\|_{\mathfrak M_p},
	\qquad n\ge0.	\]
		
		In addition, the constant coefficient is also controlled. Indeed, take $F=\mathcal H(1)$; then $\widehat F(0)=1$ and $\|F\|_{\mathcal R_p}=1$, so
		\[	|c_0|\le\|\varphi*F\|_{H^p}\le \|\varphi\|_{\mathfrak M_p}.	\]
	
		The dyadic characterization \eqref{eq7} shows that
		\[	\|\varphi\|_{H(p,\infty,1/p')}	\lesssim \|\varphi\|_{\mathfrak M_p}. 	\]
	\end{proof}
	 
\begin{theorem}\label{th3.4}
	If $1<p\le2$, then
	\begin{equation*}\label{eq:equality-small-p}
		\mathfrak M_p
		=H\left(p,\infty,\frac1{p'}\right)
	\end{equation*}
	with equivalent norms. Consequently,
	\[
	\mathcal H_g:H^p\to H^p\text{ is bounded}
	\quad\Longleftrightarrow\quad
	g\in\Lambda\left(p,\frac1p\right).
	\]
\end{theorem}

\begin{proof}
	The inclusion $\mathfrak M_p\subset H(p,\infty,1/p')$ was proved in  Theorem \ref{th3.3}. Next, we prove the reverse inclusion.
	
	Let
	\[
	A=\|\varphi\|_{H(p,\infty,1/p')},
	\qquad F=\mathcal H(f)\in\mathcal R_p,
	\qquad U=\varphi*F.
	\]
	We must show $\|U\|_{H^p}\lesssim A\|f\|_{H^p}$.
	
	If $\varphi(z)=\sum_{n\ge0}c_nz^n$, then
	\begin{align*}
		U(z)
		&=\sum_{n=0}^{\infty}c_n
		\left(\int_0^1t^nf(t)\,dt\right)z^n\\
		&=\int_0^1f(t)\sum_{n=0}^{\infty}c_n(tz)^n\,dt
		=\int_0^1f(t)\varphi(tz)\,dt.
	\end{align*}
	As before, the interchange is justified on compact subsets by uniform convergence and the Fej\'er--Riesz inequality.
	
	By Lemma \ref{lem2.3} we have
	\[
	M_p(r,\varphi')
	\lesssim A(1-r)^{-1-1/p'}.
	\]
	Differentiating under the integral  and using Minkowski's inequality we have
	\begin{align*}
		M_p(r,U')
		&\le\int_0^1t|f(t)|M_p(rt,\varphi')\,dt\\
		&\lesssim A\int_0^1
		\frac{|f(t)|}{(1-rt)^{1+1/p'}}\,dt.
	\end{align*}
	
	Set $x=1-r$, $y=1-t$, and $u(y)=|f(1-y)|$. Since
	\[	1-rt=x+y-xy,	\]
	and $\tfrac12(x+y)\le x+y-xy\le x+y$, we obtain
	\[	(1-r)^{1/p'}M_p(r,U')
	\lesssim A\,x^{1/p'}
	\int_0^1\frac{u(y)}{(x+y)^{1+1/p'}}\,dy.
	\]
	Applying Lemma  \ref{lem2.1} with $a=1/p'$ we obtain
	\begin{align*}
		\int_0^1(1-r)^{p-1}M_p(r,U')^p\,dr
		&\lesssim A^p\int_0^1|f(t)|^p\,dt\\
		&\lesssim A^p\|f\|_{H^p}^p.
	\end{align*}
	Moreover,
	\[
	U(0)=\varphi(0)\int_0^1f(t)\,dt,
	\]
	by H\"older inequality and Fej\'er--Riesz inequality  we have
	\[	|U(0)|\lesssim A\|f\|_{H^p}.	\]
	This implies that
	\[	\|U\|_{\mathcal D(p,p,1)}	\lesssim A\|f\|_{H^p}.	\]
	Since $2<p\le 2$, the embedding (\ref{eq1}) shows that 
	\[	\|U\|_{H^p}\lesssim \|U\|_{\mathcal D(p,p,1)}.	\]
	Consequently,
	\[	\|\varphi*F\|_{H^p}	\lesssim A\|f\|_{H^p}	=A\|F\|_{\mathcal R_p}.	\]
	Thus $\varphi\in\mathfrak M_p$, and the two inclusions give the asserted equality with equivalent norms.
\end{proof}

	\begin{theorem}\label{thm3.5}
		If $2<p<\infty$, then
		\begin{equation*}\label{eq:strict-big-p}
			\mathfrak M_p
			\subsetneq H\left(p,\infty,\frac1{p'}\right).
		\end{equation*}
	\end{theorem}
	
	\begin{proof}
	 It remains to exhibit a function in the growth space $H\left(p,\infty,\frac1{p'}\right)$ which fails to be a multiplier.
		Let
		\begin{equation*}\label{eq:Phi-p}
			\Phi_p(z)=\sum_{j=2}^{\infty}2^{j/p'}z^{2^j}.
		\end{equation*}
		Each dyadic block contains exactly one nonzero term, and therefore
		\[	\|\Delta_j\Phi_p\|_{H^p}=2^{j/p'}.	\]
		The dyadic characterization \eqref{eq7} gives
		\[	\Phi_p\in H\left(p,\infty,\frac1{p'}\right).	\]
	
		Let
		\[	L(z)=\log\frac{e}{1-z}	\]
		with the principal branch and set
		\begin{equation*}\label{eq:f-p}
			f_p(z)=\frac1{(1-z)^{1/p}L(z)^{1/2}}.
		\end{equation*}
		Then  $f_{p}\in \hd$. For $0<r<1$,
		\begin{align*}
	M_{p}^{p}(r,f_{p})=	\int_{-\pi}^{\pi} \frac{1}{|1-re^{i\theta}|} \left|\log\frac{e}{1-re^{i\theta}}\right |^{-\frac{p}{2}}d\theta.
		\end{align*}
		Since $\frac{p}{2}>1$,  it follows from Lemma 1 in \cite{ave} that 
	$$	\int_{-\pi}^{\pi} \frac{1}{|1-re^{i\theta}|} \left|\log\frac{e}{1-re^{i\theta}}\right |^{-\frac{p}{2}}d\theta \lesssim 1.$$
		Thus, $f_{p}\in H^{p}$.
	For $N\ge2$, let
		\[	I_N=\left[1-\frac1N,1-\frac1{2N}\right].	\]
		If $t\in I_N$, then $t^N\ge\left(1-\frac1N\right)^N\gtrsim 1$, 
		$	(1-t)^{-1/p}\asymp  N^{1/p},$
		and
		\[
		\left(\log\frac{e}{1-t}\right)^{-1/2}\asymp (\log N)^{-1/2}.
		\]
		Since $|I_N|=1/(2N)$, we obtain
		\begin{align}
			\int_0^1t^Nf_p(t)\,dt
			&\ge\int_{I_N}t^Nf_p(t)\,dt\notag\\
			&\gtrsim N^{-1}N^{1/p}(\log N)^{-1/2}\notag\\
			&=N^{-1/p'}(\log N)^{-1/2}.
			\label{eq:moment-lower}
		\end{align}
		Let $F_p=\mathcal H(f_p)\in\mathcal R_p$. The $n$-th coefficient  of $F_{p}$ is the moment $m_{n}(f_{p}):=\int^{1}_{0}t^{n}f_{p}(t)dt$. Therefore,
		\[
		\Phi_p*F_p(z)=\sum_{j=2}^{\infty}d_jz^{2^j},
		\qquad
		d_j=2^{j/p'}m_{2^j}(f_p).
		\]
		Applying \eqref{eq:moment-lower} with $N=2^j$ we obtain
		\[	d_j\gtrsim j^{-1/2}.	\]
		This implies that 
		\[\sum_{j=2}^{\infty}|d_j|^2=\infty.	\]
		By Paley's theorem, the lacunary series $\sum d_jz^{2^j}$ does not belong to $H^p$. Hence
		\[
		\Phi_p*F_p\notin H^p,
		\]
		although $F_p\in\mathcal R_p$. Therefore $\Phi_p\notin\mathfrak M_p$, and the inclusion is strict.
	\end{proof}
	
	The preceding theorem gives an explicit negative answer to the    conjecture posed in \cite{ann}.
	
	\begin{corollary}\label{cor:counterexample-symbol}
		Let $2<p<\infty$ and define
		\[
		g_p(z)=\sum_{j=2}^\infty
		\frac{2^{j/p'}}{2^j+1}z^{2^j+1}.
		\]
		Then
		\[
		g_p\in\Lambda\left(p,\frac1p\right),
		\]
		but $\mathcal H_{g_p}:H^p\to H^p$ is not bounded.
	\end{corollary} 
	 
	\begin{remark}
Theorem \ref{th3.3}-Theorem \ref{thm3.5} provide an alternative proof of Theorems 1 and 2 in \cite{ann} in terms of the multiplier space $\mathfrak M_p$.
	\end{remark}
	
		\section{Two explicit subspaces of the  multiplier space $\mathfrak M_p$}\label{sec4}

\ \ \ \ \ In this section, we aim to provide concrete sufficient conditions for a function to belong to $\mathfrak M_p$ and to identify several natural subclasses of this space. In particular, we introduce two explicit classes, $\mathcal A_p$ and $\mathcal E_p$, which correspond to different mechanisms responsible for the boundedness of the associated generalized Hilbert operators.

	\begin{proposition}\label{pro5.1}
		Let $1<q<p<\infty$. Then
		\begin{equation*}\label{eq:subcritical-Mp-embedding}
			H\left(q,\infty,\frac1{q'}\right)\hookrightarrow\mathfrak M_p.
		\end{equation*}
		More precisely,
		\[	\|\varphi\|_{\mathfrak M_p}
		\lesssim	\|\varphi\|_{H(q,\infty,1/q')}.	\]
	\end{proposition}
	
	\begin{proof}
		Let $\varphi\in H(q,\infty,1/q')$ and let
		$	A=\|\varphi\|_{H(q,\infty,1/q')}.$
		For $F=\mathcal H(f)\in\mathcal R_p$, set $U=\varphi*F$. The moment representation gives
		\[
		U(z)=\int_0^1f(t)\varphi(tz)\,dt.
		\]
		By Lemma \ref{lem2.3},
		\[
		M_q(\rho,\varphi')
		\lesssim A(1-\rho)^{-1-1/q'}.
		\]
		Therefore,  differentiating under the integral  and applying Minkowski's inequality,
		\[
		M_q(r,U')
		\lesssim A\int_0^1
		\frac{|f(t)|}{(1-rt)^{1+1/q'}}\,dt.
		\]
		Set $x=1-r$, $y=1-t$, and $u(y)=|f(1-y)|$. Since $1-rt\asymp x+y$,  applying Lemma \ref{lem2.1} with $a=1/q'$  we have
		\begin{align*}
			\int_0^1(1-r)^{p/q'}M_q(r,U')^p dr
			&\lesssim A^p\int_0^1|f(t)|^p dt\\
			&\lesssim A^p\|f\|_{H^p}^p.
		\end{align*}
		Also,
		\[
		|U(0)|
		=|\varphi(0)|\left|\int_0^1f(t)\,dt\right|
		\lesssim A\|f\|_{H^p}.
		\]
		Now, by Lemma \ref{lem2.4} we obtain
		\[	\|U\|_{H^p}
		\lesssim A\|f\|_{H^p} =A\|F\|_{\mathcal R_p}.\]
		Thus $\varphi\in\mathfrak M_p$, with the asserted norm estimate.
	\end{proof}
	
	\begin{theorem}\label{thm:A-strict}
		If $2<p<\infty$, then
		\[
		\mathcal A_p\subsetneq\mathfrak M_p.
		\]
	\end{theorem}
	
	\begin{proof}
		By Proposition \ref{pro5.1}, we know that 
		\[
		\mathcal A_p=\bigcup_{1<q<p}H\left(q,\infty,\frac1{q'}\right)
		\subset\mathfrak M_p.
		\]
		It remains to show that the inclusion is proper.
		Let
		\begin{equation*}\label{eq:Psi-p}
			\Psi_p(z)=\sum_{j=2}^{\infty}\frac{2^{j/p'}}{j}z^{2^j}.
		\end{equation*}
		Let $F=\mathcal H(f)\in\mathcal R_p$. By H\"older's inequality,
		\begin{align*}
			|m_{2^j}(f)|
			&=\left|\int_0^1t^{2^j}f(t)\,dt\right|\\
			&\le
			\left(\int_0^1|f(t)|^p\,dt\right)^{1/p}
			\left(\int_0^1t^{2^jp'}\,dt\right)^{1/p'}\\
			&\lesssim 2^{-j/p'}\|f\|_{H^p}.
		\end{align*}
		Consequently,
		\[
		\Psi_p*F(z)
		=\sum_{j=2}^{\infty}
		\frac{2^{j/p'}}j m_{2^j}(f)z^{2^j}.
		\]
		Paley's theorem shows that 
		\begin{align*}
			\|\Psi_p*F\|_{H^p}
			&\asymp
			\left(\sum_{j=2}^{\infty}
			\frac{2^{2j/p'}}{j^2}|m_{2^j}(f)|^2\right)^{1/2}\\
			&\lesssim
			\left(\sum_{j=2}^{\infty}\frac1{j^2}\right)^{1/2}
			\|f\|_{H^p}.
		\end{align*}
		Hence $\Psi_p\in\mathfrak M_p$.
		
		Now fix $1<q<p$. Since the $j$-th dyadic block of $\Psi_p$ is the monomial
		\[	\Delta_j\Psi_p(z)=\frac{2^{j/p'}}jz^{2^j}, 	\]
		we have
		\[	\|\Delta_j\Psi_p\|_{H^q}=\frac{2^{j/p'}}j.	\]
		Note that $\frac1{p'}-\frac1{q'}=\frac1q-\frac1p>0$, this implies that
		$$	2^{-j/q'}\|\Delta_j\Psi_p\|_{H^q}
		=\frac{2^{j(1/p'-1/q')}}j \longrightarrow \infty \ \ \text{as} \  j\rightarrow \infty.$$
		The dyadic characterization (\ref{eq7})	shows that $\Psi_p$ does not belong to  $H(q,\infty,1/q')$ with $q<p$, and therefore
		\[	\Psi_p\in\mathfrak M_p\setminus\mathcal A_p.	\]
		The proof is complete.
	\end{proof}

	 For $p\geq 2$, the following coefficient condition is a useful sufficient criterion.
	 
	 \begin{corollary}\label{cor:lp-block-sufficient}
	 	Let $2\leq p<\infty$ and $g(z)=\sum b_kz^k$. If
	 	\begin{equation}\label{eq:lp-block}
	 		\left(\sum_{k=2^n}^{2^{n+1}-1}|b_k|^p\right)^{1/p}
	 		=O\left(2^{-n/p'}\right),
	 	\end{equation}
	 	then $\mathcal H_g:H^p\longrightarrow H^p$ is bounded.
	 \end{corollary}
	 
	 \begin{proof}
	 For $p=2$, the condition (\ref{eq:lp-block}) is equivalent to that $g\in\Lambda(2,1/2)$, so the result is obvious. For $p>2$, by H\"{o}lder inequality we have
	 	\begin{align*}
	 		\|\Delta_ng\|_{H^2}
	 		&=\left(\sum_{k=2^n}^{2^{n+1}-1}|b_k|^2\right)^{1/2}\\
	 		&\le2^{n(1/2-1/p)}
	 		\left(\sum_{k=2^n}^{2^{n+1}-1}|b_k|^p\right)^{1/p}
	 		\lesssim2^{-n/2}.
	 	\end{align*}
	 	This shows that  $g\in\Lambda(2,1/2)$, so $g'\in H(2,\infty,1/2)\subset\mathcal A_p\subset\mathfrak M_p$.
	 \end{proof}
	 
	 \begin{remark}
	 In \cite{ann}, the authors proved that if $2\leq p <\infty$ and $\sup_{k\geq 0}k|b_{k}|<\infty$, then $\hg:H^{p}\longrightarrow H^{p}$ is bounded. This result is contained in  	Corollary \ref{cor:lp-block-sufficient}.
	 \end{remark}

	 The condition in \eqref{eq:lp-block} is not necessary, even for nonnegative coefficients. A complete elementary characterization is available under the natural additional assumption of monotonicity.
	 
	 \begin{theorem}\label{th3.9}
	 	Let $1<p<\infty$, and suppose
	 	\[	g(z)=\sum_{n=0}^\infty b_nz^n,
	 	\qquad	b_n\ge0,	\qquad
	 	b_{n+1}\le b_n. 	\]
	 	Then the following are equivalent:
	 
	 		(i) $\mathcal H_g:H^p\to H^p$ is bounded;
	 	 
	 	 (ii) $g'\in\mathfrak M_p$;
	 	
	 	(iii) $\sup_{n\ge1}n b_n<\infty$;
	 	
	 	(iv)	\[
	 		\left(\sum_{k=2^j}^{2^{j+1}-1}b_k^p\right)^{1/p}
	 		=O(2^{-j/p'}).
	 		\]
	 \end{theorem}
	 
	 \begin{proof}
	 	The equivalence of (i) and (ii) is  contained in Theorem \ref{th3.2}.
	 	
	 	(iii)$\Longleftrightarrow$(iv).	Assume  that $b_k\le C/(k+1)$. For $2^j\le k<2^{j+1}$, it is easy to see that
	 	\[	b_k\le C2^{-j},	\]
	 	and therefore
	 	\[
	 	\sum_{k=2^j}^{2^{j+1}-1}b_k^p
	 	\le C^p2^j2^{-jp}=C^p2^{-j(p-1)}.
	 	\]
	 	Taking $p$-th roots gives (iv).
	 	
	 	Conversely, suppose (iv) holds. The  monotonicity  of $b_{k}$ implies that 
	 	\[
	 	2^jb_{2^{j+1}}^p
	 	\le\sum_{k=2^j}^{2^{j+1}-1}b_k^p
	 	\lesssim2^{-j(p-1)}.
	 	\]
	 	Hence $b_{2^{j+1}}\lesssim2^{-j}$. If $2^j\le n<2^{j+1}$, then
	 	\[
	 	b_n\le b_{2^j}\lesssim2^{-j}\lesssim n^{-1}.
	 	\]
	 	Thus, (iii) holds.
	 	
	 (ii)$\Longrightarrow$(iii).	Assume $g'\in\mathfrak M_p$. By  Theorem \ref{th3.3},
	 $	g'\in H(p,\infty,1/p'). $
	 	Equivalently, $g\in\Lambda(p,1/p)$, and the dyadic characterization gives
	 	\[
	 	\left\|\sum_{k=2^n}^{2^{n+1}-1}b_kz^k\right\|_{H^p}
	 	\lesssim 2^{-n/p}.
	 	\]
	 	Choose a fixed $c>0$ so small that $|(k-2^n)\theta|\le1/4$ whenever $2^{n}\le k<2^{n+1}$ and $|\theta|\le c2^{-n}$. Then, for such $k$ and $\theta$,
	 	$$ Re \left (e^{i(k-2^{n})\theta}\right )=\cos((k-2^{n})\theta)\ge \cos\frac{1}{4}=:c_{0}>0.$$
	 	Therefore all the numbers $e^{i(k-2^{n})\theta}$	lie in a fixed sector around the positive real axis.
	 	
	 For	$|\theta|\le c2^{-n}$, we have
	 	\begin{align*}
	 		\left|\sum_{k=2^{n}}^{2^{n+1}-1}b_ke^{ik\theta}\right|
	 		&=\left|e^{i2^{n}\theta}\sum_{k=2^{n}}^{2^{n+1}-1}b_ke^{i(k-2^{n})\theta}\right|\\
	 		& \ge Re \left (\sum_{k=2^{n}}^{2^{n+1}-1}b_ke^{i(k-2^{n})\theta}\right )\\
	 		& \ge \sum_{k=2^{n}}^{2^{n}-1}b_k Re \left (e^{i(k-2^{n})\theta}\right )\\
	 		&\ge c_0\sum_{k=2^{n}}^{2^{n+1}-1}b_k\ge c_02^{n}b_{2^{n+1}}.
	 	\end{align*}
	 	Integrating over the arc $|\theta|\le c2^{-n}$ yields
	 $$	2^{-n}\bigl(2^{n}b_{2^{n+1}}\bigr)^p	\lesssim	\left\|\sum_{k=2^{n}}^{2^{n+1}-1}b_kz^k\right\|_{H^p}^p
	 		\lesssim 2^{-n}.	$$
	 	Thus,  $b_{2^{n+1}}\lesssim 2^{-n}$. Finally, by monotonicity, if $2^{n+1}\leq k \leq 2^{n+2}$,
	  then 
	  $$b_{k}\leq b_{2^{n+1}} \lesssim 2^{n} \asymp \frac{1}{k},$$
	 	which yields  $b_{k}\lesssim k^{-1}$.

	 (iii)$\Longrightarrow$(ii).
	 	Assume $b_n\le C/(n+1)$ and set
	 $ 	a_k=(k+1)b_{k+1},$
	 	so that $g'(z)=\sum_{k\ge0}a_kz^k$. Clearly, $|a_k|\le C$. We claim that the total variation of $(a_k)$ on every dyadic block is uniformly bounded. Since $b_k$ is decreasing,
	 	\begin{align*}
	 		|a_{k+1}-a_k|
	 		&=\left|(k+2)b_{k+2}-(k+1)b_{k+1}\right|\\
	 		&\le b_{k+2}+(k+1)(b_{k+1}-b_{k+2}).
	 	\end{align*}
	 	For $2^n\le k<2^{n+1}-1$, we have
	 	\[\sum_{k=2^{n}}^{2^{n+1}-1} b_{k+2}\le 2^{n}b_{2^{n}+2}\lesssim1,	\]
	 and 
	 	\[
	 	\sum_{k=2^{n}}^{2^{n+1}-2}(k+1)(b_{k+1}-b_{k+2})
	 	\le 2^{n+1}\sum_{k=2^{n}}^{2^{n+1}-2}(b_{k+1}-b_{k+2})
	 	\le 2^{n+1}b_{2^{n}+1}\lesssim1.
	 	\]
	 	Thus,
	 	\[	\sup_n\left(
	 	|a_{n}|+\sum_{k=2^{n}}^{2^{n+1}-2}|a_{k+1}-a_k|
	 	\right)<\infty.
	 	\]
	 	By Lemma \ref{lem2.2}, for every $1<q<\infty$,
	 	$$
	 		\left\|\sum_{k=2^{n}}^{2^{n+1}-1}a_kz^k\right\|_{H^q}
	 		\lesssim
	 		\left\|\sum_{k=2^{n}}^{2^{n+1}-1}z^k\right\|_{H^q}
	 		\asymp 2^{n/q'}.
	 $$
	 	Therefore, $g'\in H(q,\infty,1/q') $	for every $q>1$. If $1<p\le2$, take $q=p$ and use Theorem \ref{th3.4}. If $p>2$, choose any $q$ with $1<q<p$ and use Proposition \ref{pro5.1}. In either case $g'\in\mathfrak M_p$.
	 \end{proof}
	 
	 \begin{remark}\label{rem:positivity-alone}
  The	positivity alone  cannot replace monotonicity in  Theorem \ref{th3.9}. The symbol $g_p$ in  Corollary \ref{cor:counterexample-symbol} has nonnegative Taylor coefficients and belongs to $\Lambda(p,1/p)$, but $\mathcal H_{g_p}$ is unbounded for every $p>2$.
	 \end{remark}

	Let $p>2$ and define $t_p$ by
	\begin{equation*}\label{eq:t-p}
		\frac1{t_p}=\frac12-\frac1p,
		\qquad
		t_p=\frac{2p}{p-2}.
	\end{equation*} 
	
	We write  $\mathcal E_p:=H(p,t_{p},\frac{1}{p'})$. 
	For $p>2$, the multiplier space $\mathfrak M_p$ lies strictly below the
	endpoint coefficient space
	$$\mathfrak M_p \subsetneq H(p,\infty,\frac{1}{p'}).$$
	The strict inclusion above leaves open the question of how far the multiplier space is from the endpoint  space. The critical
	exponent
	\[	\frac1{t_p}=\frac12-\frac1p 	\]
	arises naturally from the relation between $H^p$ and $H^2$, and leads to the
	space
	\[	\mathcal E_p=H\left(p,t_p,\frac1{p'}\right).	\]
	The following result shows that this space is still too small: the multiplier
	space contains $\mathcal E_p$ strictly.

	\begin{theorem}\label{thm5.1}
		If $2<p<\infty$, then
		\[	\mathcal E_p\subsetneq\mathfrak M_p.	\]
	\end{theorem}
	
	\begin{proof}
		We first prove the continuous inclusion and then exhibit a Cauchy kernel which witnesses strictness.
		
		Let $F=\mathcal H(f)\in\mathcal R_p$. Since
		\[
		F(z)=\int_0^1\frac{f(t)}{1-tz}\,dt,
		\]
		we have
		\[
		DF(z)=(zF(z))'
		=\int_0^1\frac{f(t)}{(1-tz)^2}\,dt.
		\]
		By	Minkowski's inequality and the identity
		\[	\frac1{2\pi}\int_0^{2\pi}
		\frac{d\theta}{|1-re^{i\theta}|^2}
		\asymp\frac1{1-r^2},
		\qquad0\le a<1,	\]
		we have
		\begin{align*}
			M_1(r,DF)
			&\le\int_0^1|f(t)|
			M_1\left(r,\frac1{(1-tz)^2}\right)\,dt\\
			& \lesssim\int_0^1\frac{|f(t)|}{1-rt}\,dt.
		\end{align*}
		With $x=1-r$, $y=1-t$, the last denominator $1-rt$   is comparable to $x+y$. The case $a=0$  of Lemma \ref{lem2.2} and the Fej\'er--Riesz inequality  imply that
		\begin{equation}\label{eq:DF-estimate}
			\left(\int_0^1M_1(r,DF)^p\,dr\right)^{1/p}
			\lesssim \|f\|_{H^p}.
		\end{equation}
		
		Let $\varphi\in\mathcal E_p=H(p,t_p,1/p')$ and put $U=\varphi*F$. It obvious that
		\[	DU=\varphi*DF.	\]
		By	Young's inequality  we have
		\[
		M_p(r^2,DU)
		\le M_p(r,\varphi)M_1(r,DF).
		\]
		Changing variables $s=r^2$ in the $\mathcal D(p,2,1)$ norm and using $1-r^2\asymp1-r$, we obtain
		\begin{align*}
			\|U\|_{\mathcal D(p,2,1)}
			&\lesssim |U(0)|+
			\left\|(1-r)^{1/2}M_p(r,\varphi)M_1(r,DF)\right\|_{L^2(dr)}.
		\end{align*}
		The exponent identities
		\[
		\frac12=\frac1{t_p}+\frac1p,
		\qquad
		\frac1{p'}-\frac1{t_p}=\frac12
		\]
		show that
		\[
		\left\|(1-r)^{1/2}M_p(r,\varphi)\right\|_{L^{t_p}}
		=\|\varphi\|_{H(p,t_p,1/p')}.
		\]
		Using	H\"older's inequality with exponents $t_p/2$ and $p/2$, together with \eqref{eq:DF-estimate} we obtain 
		\begin{align*}
			\|U\|_{\mathcal D(p,2,1)}
			&\lesssim \|\varphi\|_{H(p,t_p,1/p')}	\|f\|_{H^p}.
		\end{align*}
		The constant term satisfies the same estimate directly. By (\ref{eq2}), $\mathcal D(p,2,1)\hookrightarrow H^p$ for $p>2$. Hence
		\[	\|\varphi*F\|_{H^p}
		\lesssim \|\varphi\|_{\mathcal E_p}\|F\|_{\mathcal R_p}.	\]
		This shows that $\mathcal E_p\subset\mathfrak M_p$.

		Fix $\zeta\in\mathbb T$ and set
		\[
		\kappa_\zeta(z)=\frac1{1-\overline\zeta z}.
		\]
		If $F(z)=\sum a_nz^n\in\mathcal R_p$, then
		\[
		(\kappa_\zeta*F)(z)
		=\sum_{n=0}^{\infty}\overline\zeta^{\,n}a_nz^n
		=F(\overline\zeta z).
		\]
		Thus
		\[
		\|\kappa_\zeta*F\|_{H^p}=\|F\|_{H^p}
		\lesssim\|F\|_{\mathcal R_p},
		\]
		and $\kappa_\zeta\in\mathfrak M_p$.
		
		Rotations do not change integral means, and the standard kernel estimate gives
		\[
		M_p(r,\kappa_\zeta)\asymp (1-r)^{-1/p'}.
		\]
		Therefore,
		\begin{align*}
			\|\kappa_\zeta\|_{H(p,t_p,1/p')}^{t_p}
			&\asymp\int_0^1
			(1-r)^{t_p/p'-1}(1-r)^{-t_p/p'}\,dr\\
			&=\int_0^1\frac{dr}{1-r}=\infty.
		\end{align*}
		Hence $\kappa_\zeta\notin\mathcal E_p$, this means  that the inclusion is proper.
	\end{proof}
	
	Combining the preceding results, for $p>2$ we have 
	\[
	\mathcal A_p\subsetneq\mathfrak M_p
	\subsetneq H\left(p,\infty,\frac1{p'}\right),
	\qquad
	\mathcal E_p\subsetneq\mathfrak M_p.
	\]
	
	In \cite{Blas2022}, Blasco proved that if $p>2$ and $g\in \mathcal{D}(p,t_{p},\frac{1}{p'})$, then $\hg:L^{p}(0,1)\longrightarrow H^{p}$ is bounded. Theorem \ref{thm5.1}  can also be regarded as a complement to this result.

	 \section{Elementary structure of the  multiplier space $\mathfrak M_p$ }
	 
	\ \ \ \ \  In this section, we study the elementary structure of the multiplier space  $\mathfrak M_p$. We first establish some basic functional-analytic properties of  $\mathfrak M_p$, including completeness, invariance under rotation and dilation, stability under Hadamard multiplication, and the inclusion of polynomials. We then investigate the dependence of $\mathfrak M_p$ on the exponent $p$.
	\begin{proposition}\label{prop:structure}
		For  $1<p<\infty$, the space $\mathfrak M_p$ has the following properties.

		(1) $\mathfrak M_p$ is a Banach space.
		
		(2) If $0<r<1$ and $\varphi_r(z)=\varphi(rz)$, then
			\[	\|\varphi_r\|_{\mathfrak M_p}\le\|\varphi\|_{\mathfrak M_p}.	\]
			
		(3) If $|\eta|=1$ and $\varphi_\eta(z)=\varphi(\eta z)$, then
			\[	\|\varphi_\eta\|_{\mathfrak M_p}=\|\varphi\|_{\mathfrak M_p}.	\]
			
		(4) $\mathfrak M_p$ is a module over $(H^p,H^p)$ under Hadamard multiplication. More precisely, if
			$\psi\in(H^p,H^p)$ and $\varphi\in\mathfrak M_p$, then
			\[
			\psi*\varphi\in\mathfrak M_p
			\quad\text{and}\quad
			\|\psi*\varphi\|_{\mathfrak M_p}
			\le \|\psi\|_{(H^p,H^p)}\|\varphi\|_{\mathfrak M_p}.
			\]
			
		(5) Every polynomial belongs to $\mathfrak M_p$.
	\end{proposition}
	
	\begin{proof}
	Throughout the proof,  we put
		\[
		F_0(z)=\mathcal H(1)(z)=\sum_{n=0}^{\infty}\frac{z^n}{n+1}.
		\]
		Then $F_0\in\mathcal R_p$ and $\|F_0\|_{\mathcal R_p}=\|1\|_{H^p}=1$.
		We shall repeatedly use the elementary coefficient estimate
		\begin{equation}\label{eq:Hp-coefficient-functional}
			|\widehat h(n)|\le \|h\|_{H^p},
			\qquad h\in H^p,\quad n\ge0,
		\end{equation}
		which follows by writing $\widehat h(n)r^n$ as the $n$th Fourier coefficient of
		$h(re^{i\theta})$ and then letting $r\to1^-$.
		
		\smallskip
		\noindent\emph{Proof of \textup{(i)}.}
		Let $\{\varphi_j\}_{j\ge1}$ be a Cauchy sequence in $\mathfrak M_p$, where
		\[
		\varphi_j(z)=\sum_{n=0}^{\infty}c_n^{(j)}z^n.
		\]
		Since
		\[
		(\varphi_j-\varphi_k)*F_0
		=\sum_{n=0}^{\infty}
		\frac{c_n^{(j)}-c_n^{(k)}}{n+1}z^n,
		\]
		\eqref{eq:Hp-coefficient-functional} gives, for every $n\ge0$,
		\begin{align*}
			\frac{|c_n^{(j)}-c_n^{(k)}|}{n+1}
			&\le \|(\varphi_j-\varphi_k)*F_0\|_{H^p}\notag\\
			&\le \|\varphi_j-\varphi_k\|_{\mathfrak M_p}.
			\label{eq:coefficient-Cauchy-Mp}
		\end{align*}
		Thus $\{c_n^{(j)}\}_{j\ge1}$ is a Cauchy sequence of complex numbers for each
		fixed $n$. Let
		\[
		c_n=\lim_{j\to\infty}c_n^{(j)}.
		\]
		Since a Cauchy sequence is bounded in norm, there is $C>0$ such that
		$\|\varphi_j\|_{\mathfrak M_p}\le C$ for all $j$. Applying
		\eqref{eq:Hp-coefficient-functional} to $\varphi_j*F_0$ and then passing to the
		limit yields
		\[
		|c_n|\le C(n+1),
		\qquad n\ge0.
		\]
		Consequently, the series
		\[
		\varphi(z)=\sum_{n=0}^{\infty}c_nz^n
		\]
		has radius of convergence at least one and therefore defines a function in
		$H(\mathbb D)$.
		
		Fix $F(z)=\sum_{n\ge0}a_nz^n\in\mathcal R_p$. Since
		$\{\varphi_j\}$ is Cauchy sequence in multiplier norm,
		\[
		\|(\varphi_j-\varphi_k)*F\|_{H^p}
		\le \|\varphi_j-\varphi_k\|_{\mathfrak M_p}\|F\|_{\mathcal R_p},
		\]
		so $\{\varphi_j*F\}$ is Cauchy sequence in the Banach space $H^p$. Let its limit be
		$G_F\in H^p$. Convergence in $H^p$ implies convergence of every Taylor
		coefficient. Since the $n$-th coefficient of $\varphi_j*F$ is
		$c_n^{(j)}a_n$, the $n$-th coefficient of $G_F$ is $c_na_n$. Hence
		\[
		G_F=\varphi*F.
		\]
		In particular, $\varphi*F\in H^p$ for every $F\in\mathcal R_p$, so
		$\varphi\in\mathfrak M_p$.
		
		It remains to prove convergence in the multiplier norm. Given $\varepsilon>0$,
		choose $j_0$ such that
		\[
		\|\varphi_j-\varphi_k\|_{\mathfrak M_p}<\varepsilon,
		\qquad j,k\ge j_0.
		\]
		Fix $j\ge j_0$ and $F\in\mathcal R_p$. Letting $k\to\infty$ in $H^p$ gives
		\[
		\|(\varphi_j-\varphi)*F\|_{H^p}
		\le \varepsilon\|F\|_{\mathcal R_p}.
		\]
		Taking the supremum over nonzero $F$ shows that
		\[
		\|\varphi_j-\varphi\|_{\mathfrak M_p}\le\varepsilon.
		\]
		Thus, $\mathfrak M_p$ is a  Banach space.
		
		\smallskip
		\noindent\emph{Proof of \textup{(ii)}.}
		Let $\varphi(z)=\sum_{n\ge0}c_nz^n\in\mathfrak M_p$ and
		$F(z)=\sum_{n\ge0}a_nz^n\in\mathcal R_p$. Since
		$\varphi_r(z)=\sum_{n\ge0}r^nc_nz^n$, coefficientwise multiplication gives
		\[
		(\varphi_r*F)(z)
		=\sum_{n=0}^{\infty}r^nc_na_nz^n
		=(\varphi*F)(rz).
		\]
		Radial dilation is contractive on $H^p$, and hence
		\[
		\|\varphi_r*F\|_{H^p}
		=\|(\varphi*F)_r\|_{H^p}
		\le \|\varphi*F\|_{H^p}
		\le \|\varphi\|_{\mathfrak M_p}\|F\|_{\mathcal R_p}.
		\]
		Taking the supremum over $F\ne0$ yields
		\[
		\|\varphi_r\|_{\mathfrak M_p}\le\|\varphi\|_{\mathfrak M_p}.
		\]
		
		\smallskip
		\noindent\emph{Proof of \textup{(iii)}.}
		Let $|\eta|=1$. Since
		$\varphi_\eta(z)=\sum_{n\ge0}\eta^nc_nz^n$, we have
		\[
		(\varphi_\eta*F)(z)
		=\sum_{n=0}^{\infty}\eta^nc_na_nz^n
		=(\varphi*F)(\eta z).
		\]
		Rotation is an isometry of $H^p$, so
		\[
		\|\varphi_\eta*F\|_{H^p}
		=\|\varphi*F\|_{H^p}.
		\]
		Taking the supremum over $F\ne0$ gives
		$\|\varphi_\eta\|_{\mathfrak M_p}\le\|\varphi\|_{\mathfrak M_p}$.
		Applying the same inequality to $\varphi_\eta$ and $\overline\eta$ yields the
		reverse inequality, because $(\varphi_\eta)_{\overline\eta}=\varphi$.
		Therefore
		\[
		\|\varphi_\eta\|_{\mathfrak M_p}=\|\varphi\|_{\mathfrak M_p}.
		\]
		
		\smallskip
		\noindent\emph{Proof of \textup{(iv)}.}
		Let $\psi\in(H^p,H^p)$, $\varphi\in\mathfrak M_p$, and
		$F\in\mathcal R_p$. Associativity and commutativity of the Hadamard product
		imply
		\[
		(\psi*\varphi)*F=\psi*(\varphi*F).
		\]
		Since $\varphi*F\in H^p$ and $\psi$ is a coefficient multiplier of $H^p$ into
		itself, the right-hand side belongs to $H^p$. Moreover,
		\begin{align*}
			\|(\psi*\varphi)*F\|_{H^p}
			&\le \|\psi\|_{(H^p,H^p)}\|\varphi*F\|_{H^p}\\
			&\le \|\psi\|_{(H^p,H^p)}
			\|\varphi\|_{\mathfrak M_p}\|F\|_{\mathcal R_p}.
		\end{align*}
		Taking the supremum over $F\ne0$ proves both assertions in \textup{(iv)}.
		
		\smallskip
		\noindent\emph{Proof of \textup{(v)}.}
		Let
		\[
		P(z)=\sum_{n=0}^{N}d_nz^n
		\]
		be a polynomial, and let $F=\mathcal H(f)\in\mathcal R_p$. Then
		\[
		P*F(z)=\sum_{n=0}^{N}d_nm_n(f)z^n,
		\qquad
		m_n(f)=\int_0^1t^nf(t)\,dt.
		\]
		By H\"older's inequality and the Fej\'er--Riesz inequality,
		\begin{align*}
			|m_n(f)|
			&\le \left(\int_0^1|f(t)|^p\,dt\right)^{1/p}
			\left(\int_0^1t^{np'}\,dt\right)^{1/p'}\\
			&\lesssim (n+1)^{-1/p'}\|f\|_{H^p}.
		\end{align*}
		Since $\|z^n\|_{H^p}=1$, the triangle inequality gives
		\begin{align*}
			\|P*F\|_{H^p}
			&\le \sum_{n=0}^{N}|d_n|\,|m_n(f)|\\
			&\lesssim
			\left(\sum_{n=0}^{N}|d_n|(n+1)^{-1/p'}\right)
			\|f\|_{H^p}.
		\end{align*}
		Since $\|f\|_{H^p}=\|F\|_{\mathcal R_p}$, this implies that
		$P\in\mathfrak M_p$ and also gives the explicit estimate
		\[
		\|P\|_{\mathfrak M_p}
		\lesssim \sum_{n=0}^{N}|d_n|(n+1)^{-1/p'}.
		\]
		The proof is complete.
	\end{proof} 

	 The following result shows that the  family $\{\mathfrak M_p\}_{1<p<\infty}$  is  increasing with respect to  $p$, and the inclusion is proper.
	 
	 \begin{theorem}\label{thm:Mp-strict-scale}
	 	Let $1<p_0<p_1<\infty$. Then
	 	\begin{equation*}\label{eq:Mp-strict-scale}
	 		\mathfrak M_{p_0}\subsetneq\mathfrak M_{p_1}.
	 	\end{equation*}
	 	More precisely, the inclusion is continuous:
	 	\begin{equation*}\label{eq:Mp-scale-norm}
	 		\|\varphi\|_{\mathfrak M_{p_1}}
	 		\lesssim	\|\varphi\|_{\mathfrak M_{p_0}},
	 		\qquad \varphi\in\mathfrak M_{p_0}.
	 	\end{equation*}
	 \end{theorem}
	 
	 \begin{proof}
	 	We first prove the continuous inclusion and then construct an explicit function showing that it is strict.
	 
	 	Let $\varphi\in\mathfrak M_{p_0}$. By the  growth estimate in
	 	Theorem \ref{th3.3},
	 	\begin{equation}\label{eq:Mp0-to-growth}
	 		\|\varphi\|_{H(p_0,\infty,1/p_0')}
	 		\lesssim \|\varphi\|_{\mathfrak M_{p_0}}.
	 	\end{equation}
	 	Since $p_0<p_1$, the subcritical multiplier embedding
	 	\eqref{eq:subcritical-Mp-embedding}, applied with $q=p_0$ and $p=p_1$, gives
	 	\begin{equation}\label{eq:growth-to-Mp1}
	 		\|\varphi\|_{\mathfrak M_{p_1}}
	 		\lesssim
	 		\|\varphi\|_{H(p_0,\infty,1/p_0')}.
	 	\end{equation}
	 	Combining \eqref{eq:Mp0-to-growth} and \eqref{eq:growth-to-Mp1} we have  
	 	$	\|\varphi\|_{\mathfrak M_{p_1}}
	 	\lesssim	\|\varphi\|_{\mathfrak M_{p_0}}$, and hence
	 	$\mathfrak M_{p_0}\subset\mathfrak M_{p_1}$.

	 	Choose an exponent $s$ satisfying $	p_0<s<p_1 $
	 	and define the Hadamard-gap function
	 	\begin{equation*}\label{eq:Theta-s}
	 		\Theta_s(z)=\sum_{j=1}^{\infty}2^{j/s'}z^{2^j}.
	 	\end{equation*}
	 	Next, we show that
	 $	\Theta_s\in\mathfrak M_{p_1} $ but	$\Theta_s\notin\mathfrak M_{p_0}.	$
	 	
	 As in the proof of Theorem  \ref{thm3.5},  it is easy to check that 
	 	$\label{eq:Theta-growth-s}
	 		\Theta_s\in H\left(s,\infty,\frac1{s'}\right).
	 $
	  Since $s<p_1$, the 
	 	embedding \eqref{eq:subcritical-Mp-embedding} now yields
	 	\[
	 	\Theta_s\in\mathfrak M_{p_1}.
	 	\]
	 	
	 	It remains to show that $\Theta_s\notin\mathfrak M_{p_0}$. By
	 	Theorem \ref{th3.3}, membership in $\mathfrak M_{p_0}$ would imply
	 	membership in $H(p_0,\infty,1/p_0')$. For $N\ge1$, set
	 	\[
	 	r_N=1-2^{-N}.
	 	\]
	 	The coefficient estimate $|\widehat h(n)|r^n\le M_{p_0}(r,h)$, applied to the
	 	coefficient of $z^{2^N}$ in $\Theta_s$ we obtain
	 	\[
	 	M_{p_0}(r_N,\Theta_s)
	 	\ge 2^{N/s'}r_N^{2^N}.
	 	\]
	 	Since $(1-2^{-N})^{2^N}\gtrsim 1$, it follows that
	 $$
	 		(1-r_N)^{1/p_0'}M_{p_0}(r_N,\Theta_s)
	 		\gtrsim 2^{-N/p_0'}2^{N/s'}
	 		=2^{N(1/s'-1/p_0')}.
	 $$
	 	Because $s>p_0$, we have
	 	\[	\frac1{s'}-\frac1{p_0'}=\frac1{p_0}-\frac1s>0.	\]
	 	This implies that 
	 	$$2^{N(1/s'-1/p_0')}\longrightarrow \infty \ \ \ \text{as}\ N\rightarrow \infty.$$
	 	 Hence
	 	\[	\Theta_s\notin H\left(p_0,\infty,\frac1{p_0'}\right),	\]
	 	and therefore $\Theta_s\notin\mathfrak M_{p_0}$. This  completes the proof.
	 \end{proof}

	 As a consequence, the multiplier spaces form a strictly increasing scale:
	 \[
	 \mathfrak M_{p_0}\subsetneq\mathfrak M_{p_1}
	 \qquad(1<p_0<p_1<\infty).
	 \]
	 For $1<p\le2$ this agrees with the exact identification
	 $\mathfrak M_p=H(p,\infty,1/p')$. For $p>2$, the scale continues to increase,
	 but every $\mathfrak M_p$ remains a proper subspace of the endpoint growth
	 space $H(p,\infty,1/p')$.

	\section{Final remarks}
	
\ \ \ \  \ 	The space  $	\mathfrak M_p=(\mathcal R_p,H^p)$
is the exact symbol-derivative space for generalized Hilbert operators on $H^p$. For $1<p\le2$ it is the classical mixed growth space $H(p,\infty,1/p')$. For $p>2$ it is genuinely smaller, but it contains both the class $\mathcal A_p$ and the class $\mathcal E_p$.
	
There remain several natural open problems for further investigation. For $p>2$, it would be desirable to obtain an intrinsic characterization of $\mathfrak M_p$, possibly through a vector-valued Carleson measure condition, a tent-space norm, or an atomic decomposition.  It is also natural to investigate the compact multiplier space associated with $\mathfrak M_p$, including a characterization of compactness and a formula for the essential norm.


\section*{Conflicts of Interest}
The authors declare that there is no conflict of interest.

\section*{Funding}
The research is supported by the Natural Science Foundation of Hunan Province of China (No.
2025JJ60064) and the Scientific Research Fund of the Hunan Provincial Department of Education (Nos.24B0861 and 24C0222).




\section*{Availability of data and materials}
Data sharing not applicable to this article as no datasets were generated or analysed during
the current study: the article describes entirely theoretical research.




\pdfbookmark[1]{ References}{4}

\begin{thebibliography}{99}



 \bibitem{ave}
K. Avetisyan, A note on mixed norm spaces of analytic functions, Aust. J. Math.
Anal. Appl.  9(1) (2012) 1--6.



\bibitem{bh}
G. Bao, H. Wulan, Hankel matrices acting on Dirichlet spaces, J. Math. Anal. Appl. 409 (2014) 228--235.


\bibitem{bell}
C. Bellavita, V. Daskalogiannis, G. Stylogiannis,  On the Hilbert matrix operator: A brief survey. In:  J. Ball,  H. Tylli, J.  Virtanen,(eds) Operator Theory, Related Fields, and Applications. IWOTA 2023. Operator Theory: Advances and Applications, vol 307. Birkh\"{a}user, Cham. 







\bibitem{Blas2022}
O.~Blasco, Remarks on generalized Hilbert operators,
J. Math. Sci. 266  (2022) 274--284.






\bibitem{bok}
V. Bo\v{z}in, B. Karapetrovi\'{c}, Norm of the Hilbert matrix on Bergman spaces, J. Funct. Anal. 274 (2018) 525--543.







\bibitem{dia1}
E. Diamantopoulos, A. Siskakis, Composition operators and the Hilbert matrix, Studia Math. 140 (2000) 191--198.

\bibitem{dia2}
E. Diamantopoulos, Hilbert matrix on Bergman spaces, Illinois J. Math. 48(3) (2004) 1067--1078.




\bibitem{dos}
M. Dostani\'{c}, M. Jevti\'{c}, D. Vukoti\'{c}, Norm of the Hilbert matrix on Bergman and Hardy spaces and theorem of Nehari type, J. Funct. Anal. 254 (2008) 2800--2815.



\bibitem{b1}
P. Duren, Theory of $H^{p}$ spaces, Academic {P}ress, New {Y}ork, 1970.

\bibitem{f2}
T. Flett, The dual of an inequality of Hardy and Littlewood and some related inequalities, J. Math. Anal. Appl. 38(3) (1972) 746--765.

\bibitem{ann}
P. Galanopoulos, D. Girela,  J.   Pel\'{a}ez, A.Siskakis, Generalized Hilbert operators,  Ann. Acad. Sci. Fenn. Math. 39(1) (2014) 231--258.

\bibitem{rra}
P. Galanopoulos1,  D. Girela, Operators of Hilbert and Ces\`{a}ro type acting on Dirichlet
spaces, Rev. Real Acad. Cienc. Exactas Fis. Nat. Ser. A-Mat. 119 (2025) Paper No. 34.



\bibitem{hl}
G. Hardy, J. Littlewood, Some properties of fractional integrals II,  Math.  Z.  34 (1932) 403--439.


\bibitem{jk}
M. Jevti\'{c}, B. Karapetrovi\'{c}, Hilbert matrix on spaces of Bergman-type, J. Math. Anal. Appl.453 (2017) 241--254.



\bibitem{jev}
M. Jevti\'{c},  D. Vukoti\'{c}, M. Arsenovi\'{c}, Taylor Coefficients and Coefficient Multipliers of Hardy and
Bergman-Type Spaces, vol. 2. Springer, Cham 2016.



\bibitem{lnp}
B. {\L}anucha, M. Nowak, M. Pavlovi\'{c}, Hilbert matrix operator on spaces of analytic functions, Ann. Acad. Sci. Fenn. Math. 37 (2012) 161--174.

\bibitem{ml}
M. Lindstr\"{o}m, S. Miihkinen, N. Wikman, Norm estimates of weighted composition operators pertaining to the Hilbert
matrix, Proc. Am. Math. Soc.  147 (2019) 2425--2435.

\bibitem{mat}
M. Mateljevi\'{c},  M. Pavlovi\'{c}, $L^{p}$ behaviour of the integral means of analytic functions,  Studia Math. 77  (1984)  219--237.


\bibitem{pj}
J. Pel\'{a}ez, J.  R\"{a}tty\"{a}, Generalized Hilbert operators on weighted
Bergman spaces, Adv. Math.  240 (2013) 227--267.



\bibitem{pm}
M. Pavlovi\'{c}, Function classes on the unit disc,  An Introduction, De Gruyter Studies in
Mathematics, vol. 52, De Gruyter, Berlin, 2014.






\bibitem{tt}

P. Tang, Operators of Hilbert type acting on some spaces of analytic functions,  Rev. Real Acad. Cienc. Exactas Fis. Nat. Ser. A-Mat. to appear.




\bibitem{zg}
A.~Zygmund, Trigonometric {S}eries, Vol. II.  Cambridge {U}niversity {P}ress, London,
1959.

\end{thebibliography}
 \end{document}